\documentclass[11pt]{amsart}
\usepackage{amsmath}
\usepackage{amsthm}
\usepackage{amscd}
\usepackage{amssymb}
\usepackage{amsfonts}
\usepackage{graphics}
\usepackage{graphicx}
\usepackage{color}
\usepackage[dvipsnames]{xcolor}
\usepackage{comment}
\usepackage[normalem]{ulem}
\date{}

\usepackage{enumitem}

\usepackage{upgreek}

\usepackage{mathrsfs}

\usepackage{cancel}

\newlength{\defbaselineskip }
 \long\def\salta#1{\relax}

\usepackage{hyperref}

 \theoremstyle{plain}
\newtheorem{theorem}{Theorem}[section]
\newtheorem{proposition}[theorem]{Proposition}
\newtheorem{lemma}[theorem]{Lemma}

\theoremstyle{definition}
\newtheorem{definition}[theorem]{Definition}

\newtheorem{remark}[theorem]{Remark}

\newcommand{\re}{\mathbb{R}}

\newcommand{\dis}{\displaystyle}

\def\wr{W^{-1,r'}(\Omega)}
\def\t1p0{T^{1,p}_{0}(\Omega)}

\def\m2{M^{\frac{N(p-1)}{N-1}}(\Omega)}

\def\sobr{W^{1,r}_{0}(\Omega)}

\def\into{\int_{\Omega}}

\def\w-1p'{W^{-1,p'}(\Omega)}
\def\pw-1p'u{L^{p'}(0,1;W^{-1,p'}(\Omega))}

\def\lp'n{(L^{p'}(\Omega))^{N}}

\newcommand{\rnk}[1]{\operatorname{\mathrm{rank}}{\,#1}}

\title[Anisotropic quasilinear elliptic systems]{Topological invariants for anisotropic quasilinear elliptic systems:\\ A Poincar\'e-Hopf formula}
\author[N. Borgia]{Natalino Borgia}
\address{Dipartimento di  Matematica  \\ Universit\`{a} degli Studi di Bari Aldo Moro \\ Via Orabona 4\\ 70125 Bari, Italy}

\begin{document}

\begin{abstract}
We consider the functional  $I_{\delta,\Psi_1,\Psi_2}:X \to \mathbb{R}$ defined for any $z=(u,v) \in X$ as
\begin{align*}
	I_{\delta,\Psi_1,\Psi_2}(z) 
	& =  \into  \Psi_1(\nabla u) \, dx + \into \Psi_2(\nabla v ) \, dx  \\
	& \quad - \into
	H(\delta,x,u(x),v(x)) \,dx,    
\end{align*}
where $\Omega$ is a smooth bounded domain of $\mathbb{R}^N$, $\Psi_1, \Psi_2: \mathbb{R}^N \to \mathbb{R}$ are convex functions satisfying suitable conditions for $1 < p, q < N$, and the nonlinearity $H$ is allowed to exhibit both subcritical and critical growth and may depend on a parameter $\delta \in I \subseteq \mathbb{R}$. Here, the functional space setting is given by the product space $X :=W_0^{1,p}(\Omega)\times
W_0^{1,q}(\Omega)$ equipped with the norm
$\|z\|= \|u\|_{1,p} + \|v\|_{1,q}$ for any 
$z=(u,v)\in X$, where $\| \cdot \|_{1,s}$ is the usual  norm in $W^{1,s}_0(\Omega)$. 
In this paper we prove that $I_{\delta,\Psi_1,\Psi_2}'$ is of class $(S)_+$ and,  from \cite[Theorem 1.1]{CD1}, we infer that each isolated 
critical point of $I_{\delta,\Psi_1,\Psi_2}$ has critical
 groups of finite type and a  Poincar\'e-Hopf formula holds.
\end{abstract}



\date{\today}

\maketitle

\section{INTRODUCTION}

\medskip
\noindent
In the context of Hilbert spaces, the celebrated Poincaré–Hopf formula has been extended to functionals $f$ whose gradient $\nabla f$ is a compact perturbation of the identity. This framework allows one to employ the Leray–Schauder degree theory rather than the classical Brouwer degree for $C^1$ functions $f \colon \mathbb{R}^N \to \mathbb{R}$ (see, e.g., \cite[Theorem~8.5]{mawhin_willem1989}, \cite[Theorem~3.2]{chang}, \cite[Theorem~3.2]{li_li_liu2005}, \cite[Theorem~1.3]{bartsch_dancer2009}, and \cite{benci1991}).

Conversely, if $Y$ is a Banach space (not necessarily Hilbert), the derivative $f'$ is naturally defined from $Y$ into its dual $Y'$, which departs from the standard setting of the Leray–Schauder degree. In \cite{kim_wang1989}, Kim and Wang obtained a Poincaré–Hopf type result for the specific case where the Banach space is densely embedded in a Hilbert space, and the functional $f$ induces a vector field $\nabla f$ on the Banach space via the inner product of the Hilbert space. They achieved this by applying extensions of the Leray–Schauder degree for maps acting from the Banach space into itself.

Recently, in \cite{CD1}, Cingolani and Degiovanni established a Poincaré–Hopf formula in a general reflexive Banach space $Y$. Their main approach was to replace the Leray–Schauder degree with the Browder degree developed in \cite{browder1983} for demicontinuous maps (i.e., continuous from the strong topology of the domain to the weak topology of the codomain) that satisfy the $(S)_+$ condition, meaning that:

\begin{definition} 
Let $(Y, \lVert \cdot \rVert)$ be a reflexive Banach space, whose dual is denoted by $Y'$. Let $D\subseteq Y$ and let us consider a map $F:D\longrightarrow Y'$. $F$ is said to be \emph{of class~$(S)_+$} if, for every sequence $\{w_k\}_k$ in $D$ weakly convergent to $w$ in $Y$ with
		\[
		\limsup_{k \to + \infty} \,\langle F(w_k),w_k-w\rangle \leq 0\,,
		\]
we have $\|w_k-w\|\to 0$.
\end{definition}

We notice that if $Y$ is a Hilbert space and $Y'$ is identified with $Y$ in
the standard way, then any compact perturbation of the identity
is of class~$(S)_+$. However, the class~$(S)_+$ allows one to consider further
interesting cases, also in the Hilbert setting.

As noticed in~\cite{CD1,degiovanni2009}, any compact perturbation of a map of
class~$(S)_+$ is still of class~$(S)_+$. 
Therefore, the class includes the perturbations of the
$p$-Laplace operator by terms with subcritical growth. 
We refer to~\cite[Theorem~10]{dinca_jebelean_mawhin2001}, where $-\Delta_p:W^{1,p}_0(\Omega)\longrightarrow W^{-1,p'}(\Omega)$
is continuous and of class~$(S)_+$. 
In~\cite{CD1}, the case with critical growth has been covered, which is interesting also for $p=2$.
We also mention the paper~\cite{ad}, where Almi and Degiovanni showed how the degree for maps of class~$(S)_+$ can be used to
define, by a suitable approximation technique, a degree for quasilinear elliptic
equations with natural growth conditions.

In this paper we  prove the finiteness of the critical groups and a local Poincar\'e-Hopf type formula for the Euler functional associated to the following system
\begin{equation}\label{AnisotropicSystem}
	\begin{cases}
		\begin{array}{ll}
			-\text{\rm div}  \; [ \nabla \Psi_1 \left( \nabla u  \right)] = H_s(\delta,x, u,v) & \hbox{in} \ \Omega,
			\medskip \\
			-\text{\rm div} \; [ \nabla \Psi_2 \left( \nabla v  \right)]= H_t(\delta,x,u,v) & \hbox{in} \ \Omega, \medskip\\
			u=v=0  & \hbox{on} \ \partial\Omega,
		\end{array}
	\end{cases}
\end{equation}
where $\Omega$ is a smooth bounded domain of $\mathbb{R}^N$,  $ N \geq 2 $.\\
Denoting by $I$ an interval of $\mathbb{R}$, the  function $H:I\times \Omega \times \re^2 \to \mathbb{R}$  is such that $H(\delta,x,\cdot,\cdot) \in C^1(\re^2, \re)$ for any $\delta \in I$ and for a.e. $x\in \Omega$, $H_s(\delta, \cdot,s,t)$ and $H_t(\delta, \cdot,s,t)$ are measurable for every $(\delta,s,t) \in I\times \re^2$, and $H(\delta, \cdot ,0,0)$ belongs to $ L^{1}(\Omega)$ for any $\delta \in I$.\\
Moreover, we assume:
\begin{itemize}
\item[$(\Psi)$] the functions $\Psi_{1}, \Psi_{2}: \mathbb{R}^N \to \mathbb{R}$ are of class $C^1$ with $\Psi_1(0)=0$, $\nabla \Psi_1(0)=0$  and $\Psi_2(0)=0$, $\nabla \Psi_2(0)=0$. Moreover, given $\gamma \geq 0,$  $r >1$ and denoting by $\Psi_{r,\gamma}: \mathbb{R}^N \to \mathbb{R}$ the function defined as
\begin{equation}\label{rArea}
	\displaystyle \Psi_{r,\gamma}(\xi):= \frac{1}{r} \biggl[ \left( \gamma^2 +
	\left| \xi \right|^2\right)^{ \frac{r}{2}} - \gamma^r \biggr] ,
\end{equation}	
we assume that there exist $\alpha \geq 0$, $1<p<N$ and $ \frac{1}{p} < \nu_1 \leq C_1$ such that $ \displaystyle \left( \Psi_1 - \nu_1 \Psi_{p,\alpha}    \right)$  and  $ \displaystyle \left( C_1 \Psi_{p,\alpha} - \Psi_1   \right)$  are both convex, and there exist $\beta \geq 0$, $1<q<N$ and $ \frac{1}{q} < \nu_2 \leq C_2$ such that $ \displaystyle \left( \Psi_2 - \nu_2 \Psi_{q,\beta}    \right)$ and $ \displaystyle \left( C_2 \Psi_{q,\beta}   - \Psi_2 \right)$ are both convex.\\
\item[$(H)$]
	there exists $C>0$ such that
	\[|H_s(\delta,x,s,t)|\leq C \left( 1 + |s|^{p^*-1} +|t|^{q^*\frac{p^*-1}{p^*}}   \right)
	\]
	\[|H_t(\delta,x,s,t)|\leq C \left( 1 + |s|^{p^*\frac{q^*-1}{q^*}}+ |t|^{q^*-1}  \right),
	\]
	for a.e. $x\in \Omega$ and every $(\delta,s,t)\in I\times \re^2$, where $p^*:=Np/(N-p)$ and $q^*:=Nq/(N-q)$ are the critical Sobolev exponents of $p$ and $q$ respectively, where $p$ and $q$ are introduced in assumption $(\Psi)$.

\end{itemize}

\medskip
\noindent
The space setting is the product space $X:= W_0^{1,p}(\Omega)\times W_0^{1,q}(\Omega)$ equipped with the norm
\begin{equation}\label{normaX}
\|z\|= \|u\|_{1,p} + \|v\|_{1,q}, \quad z=(u,v)\in X,
\end{equation}
where  $\| \cdot \|_{1,r}$ denotes the usual gradient norm in $W^{1,r}_0(\Omega)$. We will also denote by $\| \cdot \|_{r}$ the usual norm in $L^r(\Omega)$. The dual space of $W^{1,r}_0(\Omega)$ will be denoted by $W^{-1,r'}(\Omega)$, while the dual space of $X$ will be denoted by $X'$.

Condition $(\Psi)$ was introduced in \cite{CDV}, where Cingolani, Degiovanni and Vannella took into account the scalar version of system \eqref{AnisotropicSystem}, considering a nonlinearity that is independent of the parameter $\delta$ and is allowed to exhibit critical growth.  In particular, assumption $(\Psi)$ involves various anisotropic operators. For instance, we mention the interesting papers \cite{ACCFM,ACF} in which, letting $ B(t):= t^p/p \text{ for } t>0,$  Antonini, Cianchi, Ciraolo, Farina and Maz'ya considered  
$$ \displaystyle  \Psi_1=B \circ \mathscr{H},$$
where the norm $ \mathscr{H}: \mathbb{R}^N \to \mathbb{R}$ (different from the Euclidean norm) is of class $C^2$ in $\mathbb{R}^N \setminus \{O \} $, and  its anisotropic unit ball is uniformly convex (see for instance the paper \cite{CFV} of Cozzi, Farina and Valdinoci). Notice that if $\mathscr{H}$ is the Euclidean norm $\mathscr{H}(\xi)=|\xi|$, then we are in the isotropic case, and choosing $ B(t):= t^p/p \text{ for } t>0,$  we get $-\text{\rm div}  \; [ \nabla \Psi_1 \left( \nabla u  \right)]= -\Delta_p u$, hence we recover the $p$-Laplacian.  From a mathematical point of view, the study of anisotropic operators is interesting since the lack of isotropy produces a richer geometric structure. From a physical point of view, anisotropic operators and setting arise in different contexts, such as surface energy (see e.g. \cite{GIGA}) crystallography (see e.g. \cite{TAYLOR1,TAYLOR2,WULFF}), thermodynamics (see e.g. \cite{GURTIN}), noise-removal procedures in digital image processing (see e.g. \cite{ESE}).

It is immediate to see that, for any $1<p<N,$ $\alpha \geq 0$ and $1<q<N,$ $\beta \geq 0$, the functions $\Psi_1 =\Psi_{p,\alpha}$ and $\Psi_2 =\Psi_{q,\beta}$ satisfy assumption $(\Psi)$ (it is sufficient to take $\nu_1=C_1=1$ and $\nu_2=C_2=1$, respectively).\\
Therefore, considering that
\begin{equation*}
\displaystyle  
\nabla \Psi_{r,\gamma}(\xi)=
\begin{cases}
\left( \gamma^2 +
	\left| \xi \right|^2\right)^{ \frac{r-2}{2}} \xi  & \text{ if } \xi \neq 0, \bigskip \\
0 & \text{ if } \xi=0,
\end{cases}
\end{equation*}
the study of system \eqref{AnisotropicSystem} includes systems involving the $p$- and $q$-Laplacian operators, or $p$-area and $q$-area type operators, that is
\begin{equation*}
	\begin{cases}
		\begin{array}{ll}
			-\text{\rm div} \left( (\alpha^2+|\nabla u|^2)^{\frac{p-2}{2}}\nabla u\right) = H_s(\delta,x, u,v) & \hbox{in} \ \Omega;
			 \medskip \\
			-\text{\rm div} \left( (\beta^2+|\nabla v|^{2})^{\frac{q-2}{2}}\nabla v\right)= H_t(\delta,x,u,v) & \hbox{in} \ \Omega; \medskip \\
			u=v=0  & \hbox{on} \ \partial\Omega.
		\end{array}
	\end{cases}
\end{equation*}
For existence results, we mention for example \cite{ABC,BCV,CSS,denapolimarani,GPZ,MP}.\\
Systems involving this kind of quasilinear operators 
model some phenomena in non-Newtonian mechanics, nonlinear elasticity and glaciology, combustion theory, population biology (see \cite{boccardodefiguerido,diazthelin, glow, manamaw,marcellini1}).

Assumptions $(\Psi)$ and $(H)$ and Nemytskii operator theory give that the Euler functional $I_{\delta,\Psi_1,\Psi_2}:X \to \mathbb{R}$ associated to \eqref{AnisotropicSystem} and defined as
\begin{align}\label{funzionale0}
	I_{\delta,\Psi_1,\Psi_2}(z) :=  \into  \Psi_1(\nabla u) \, dx + \into \Psi_2(\nabla v ) \, dx 
	- \into H(\delta,x,u,v) \, dx, 
\end{align}
is of class $C^1$ on $X$.

As in \cite{CD1} for the $p$-Laplace operator (see also \cite{ad,CDV}), we show that also the case with critical growth for the functional $I_{\delta,\Psi_1,\Psi_2}$ gives rise locally to an operator of class $(S)_+$. 
However, in our setting we have to face some additional difficulties in order to show that the Fréchet derivative $I'_{\delta,\Psi_1,\Psi_2}$ of the Euler functional associated to a quasilinear system is locally an operator  of class $(S)_+$. About the nonlinearity $H$, in a quasilinear system setting  new difficulties arise in computations due to the coupling of $u$ with $v$, and retracing the idea contained in \cite{BCV2},
 we are able to manage them also when $H$ is allowed to grow critically, and depending on $x \in \Omega$ and $\delta \in I$. 

Precisely, we will prove the following:

\begin{theorem}\label{IprimodeltaPsi1Psi2S+}
Let  $H:I\times \Omega \times \re^2 \to \mathbb{R}$  be a nonlinearity satisfying assumption $(\mathcal{H})$, and let $\Psi_{1}, \Psi_{2}: \mathbb{R}^N \to \mathbb{R}$ two functions satisfying $(\Psi)$. Considering the $C^1$-functional $I_{\delta,\Psi_1,\Psi_2}$ defined in \eqref{funzionale0}, there exists a radius $R:=R(N,p,q,C,\nu_1,\nu_2)$ such that, for every $z_0=(u_0,v_0) \in X$, the map $I'_{\delta,\Psi_1,\Psi_2}:X \to X'$ is of class $(S)_+$ on 
$$ \displaystyle \overline{B_R(z_0)}:=\left\lbrace z \in X \; : \; \lVert z - z_0 \rVert \leq R \right\rbrace.$$
\end{theorem}

Let us now recall that  for any $q \in \mathbb{N}$, the $q$-th critical group of a functional $f$ at $\bar{u} \in \text{Dom} f$ is defined as
$$ \dis
C_q(f,\bar{u}) =
H_q\left(\{f\leq f(\bar{u})\},\{f\leq f(\bar{u})\}\setminus\{\bar{u}\}\right)\,,
$$
where $H_*$ denotes singular homology (see e.g.~\cite{chang1981,chang, mawhin_willem1989,rybakowski1987,spa}).

As a straightforward consequence of Theorem \ref{IprimodeltaPsi1Psi2S+} and \cite[Theorem 1.1]{CD1},   we also obtain the finiteness of the critical groups and a local Poincar\'e-Hopf  formula for the functional $I_{\delta,\Psi_1,\Psi_2}$, that is:

\begin{theorem}\label{TheorFinGrupCrit}
If $\bar{z}$ is an isolated critical point of the functional $I_{\delta,\Psi_1,\Psi_2}$ defined in \eqref{funzionale0}, then $C_*(I_{\delta,\Psi_1,\Psi_2},\bar{z})$ is of finite type and we have
\begin{equation*}
\mathrm{deg}(I_{\delta,\Psi_1,\Psi_2}',V,0) = \sum_{m\geq 0}
(-1)^m\,\rnk{C_m(I_{\delta,\Psi_1,\Psi_2},\bar{z})}
\end{equation*}
for every sufficiently small neighborhood $V$ of $\bar{z}$.
\end{theorem}

\noindent
However, for the reader's convenience, in Section \ref{finitezzagruppicritici} we retrace the main ideas from the proof of \cite[Theorem 1.1]{CD1} that yield the previous result.

Concerning the computation of the critical groups, we mention \cite{CCMV}  when $\Psi_1=\Psi_{p,\alpha}$ and $\Psi_2=\Psi_{q,\alpha}$ with  $\alpha=\beta \geq 0$ and $2 < p,q <N$. The computation of the critical groups is needed in \cite{BCV} to obtain Amann-Zender type results for quasilinear systems. It is still an open problem the evaluation of the critical groups when $p$ or $q$ are less than $2$, or $\alpha \neq \beta$. We also refer the reader to   \cite{AFTPAC,ASSCING,CD,CDV,CDS,CV,CV2} for critical groups estimates for quasilinear elliptic equations (see also \cite{PAO}).\\

\section{PRELIMINARIES}

\medskip
\noindent
Let us begin by analyzing assumption $(\Psi)$. Let $\Psi: \mathbb{R}^N \to \mathbb{R}$ be a function of class $C^1$ with $\Psi(0)=0$, $\nabla \Psi(0)=0$, and let us assume that there exist $r>1$, $ \gamma \geq 0$ and $ \frac{1}{r} < \nu \leq C$ such that $ \displaystyle \left( \Psi - \nu \Psi_{r,\gamma}    \right)$  and  $ \displaystyle \left( C \Psi_{r,\gamma} - \Psi   \right)$  are both convex. Then $\Psi$ is strictly convex and for any $\xi \in \mathbb{R}^N$ satisfies
\begin{equation}\label{disuguaglianza1convessità}
\displaystyle \nu \Psi_{r,\gamma}( \xi) \leq \Psi (\xi) \leq C \Psi_{r,\gamma}( \xi) 
\end{equation}
and
\begin{equation}\label{disuguaglianza2convessità}
\displaystyle \nu( \gamma^2 + \left| \xi \right|^2   )^{\frac{r-2}{2}} |\xi|^2 \leq  \nabla \Psi ( \xi ) \cdot \xi  \leq  C( \gamma^2 + \left| \xi \right|^2   )^{\frac{r-2}{2}} |\xi|^2.
\end{equation}
Moreover, the number $r$ for which $\Psi$ satisfies assumption $(\Psi)$ is unique.

The Euler functional associated with system \eqref{AnisotropicSystem} is $I_{ \delta, \Psi_1, \Psi_2}: X \to \mathbb{R},$ defined for any $z=(u,v) \in X$ as
\begin{align}\label{funzionale}
	I_{\delta,\Psi_1,\Psi_2}(z) :=  \into  \Psi_1(\nabla u) \, dx + \into \Psi_2(\nabla v ) \, dx 
	- \into H(\delta,x,u,v) \, dx. 
\end{align}
By assumption  $(H)$ and Nemytskii operator theory, the part of \eqref{funzionale} involving the nonlinearity $H$ is of class $C^1$ on $X$. Assumption $(\Psi)$ implies that the part of \eqref{funzionale} involving the anisotropic-type operators is also of class $C^1$ on $X$ (see \cite[Appendix B]{BT} for the proof). Therefore, the functional $I_{\delta, \Psi_1, \Psi_2 }$ is of class $C^1$ on $X$, and for any $z_0=(u_0,v_0) \in X$ and $z=(u,v) \in X$ we have
\begin{align*}
	\langle I'_{\delta,\Psi_1,\Psi_2}(z_0), z \rangle = & \displaystyle  \into \nabla \Psi_1 \left( \nabla u_0 \right) \cdot \nabla u \, dx +\into
	\nabla \Psi_2 \left( \nabla v_0  \right) \cdot \nabla v \,dx \\
	&-\displaystyle\into   \bigl(H_s(\delta,x,u_0,v_0) u +  H_t(\delta,x,u_0,v_0) v\bigr)\, dx. \nonumber
\end{align*}

We begin recognizing that the part involving  the anisotropic-type operators  is of class $(S)_+$. Let us consider the map $\mathcal{H}_{\Psi}\!:\!\sobr\to \wr$ defined for any $w \in \sobr$ by
$$ \dis \mathcal{H}_{\Psi}(w):=-\text{\rm div}  \; [ \nabla \Psi \left( \nabla w  \right)], $$
where the map $\mathcal{H}_{\Psi}(w): \sobr \to \mathbb{R}$ acts on any $\phi \in \sobr$ as
$$ \dis \langle \mathcal{H}_{\Psi}(w), \phi \rangle := \into \nabla \Psi \left( \nabla w \right) \cdot \nabla \phi \, dx.    $$

\begin{lemma}\label{s+}
The map $H_{\Psi}\!:\!\sobr\to \wr$ 
is of class $(S)_+$.
\end{lemma}

\begin{proof}
Let $\{w_k\}_k \subset \sobr$ and $w \in \sobr$ such that $w_k \rightharpoonup w$ in $\sobr$ and
\begin{equation}\label{ipotesi111}
 \displaystyle \limsup_{k \to + \infty} \, \langle  \mathcal{H}_{\Psi}(u_k), u_k - u  \rangle \leq 0.
\end{equation}
Explicitily, we have 
$$ \displaystyle \langle  \mathcal{H}_{\Psi} (w_k), w_k - w  \rangle = \into \nabla \Psi \left( \nabla w_k \right) \cdot \nabla \left( w_k - w \right) \, dx.$$
Since $w_k \rightharpoonup w$ in $\sobr$, we have
\begin{equation*}
\displaystyle \into  \nabla \Psi (\nabla w) \cdot \nabla \left(w_k - w \right) \, dx  \to 0
\end{equation*}
and
\begin{equation*}
\displaystyle \into  \nabla \Psi_{r,\gamma} (\nabla w) \cdot \nabla \left(w_k - w \right) \, dx  \to 0.
\end{equation*}
By convexity of the function $\displaystyle f:=\left( \Psi - \nu \Psi_{r,\gamma}    \right)$, for any $ \xi_1, \xi_2 \in \mathbb{R}^N$ we have
$$ \displaystyle \left( \nabla f(\xi_2)  - \nabla f(\xi_1) \right) \cdot \left( \xi_2 - \xi_1  \right) \geq 0,$$
whence
$$ \displaystyle \left( \nabla \Psi(\xi_2)  - \nabla \Psi(\xi_1) \right) \cdot \left( \xi_2 - \xi_1  \right) \geq \nu \left( \nabla \Psi_{r,\gamma}  (\xi_2)  - \nabla \Psi_{r,\gamma}  (\xi_1) \right) \cdot \left( \xi_2 - \xi_1  \right).$$
Therefore,
\begin{align*}
\displaystyle
     & \nu \into \left( \nabla \Psi_{r,\gamma}  (\nabla w_k)  - \nabla \Psi_{r,\gamma}  (\nabla w) \right) \cdot \nabla \left(   w_k - w  \right) \, dx \\
& \qquad \qquad \leq  \into \left( \nabla \Psi  ( \nabla w_k)  - \nabla \Psi  (\nabla w) \right) \cdot \nabla \left(  w_k - w  \right) \, dx \nonumber.
\end{align*}
Altogether, and considering also the properties of $\dis \limsup$, we infer that 
\begin{equation*}
\displaystyle \nu \, \limsup_{k \to + \infty} \, \langle  H_{\Psi_{r,\gamma}}(w_k), w_k - w    \rangle \leq \limsup_{k \to + \infty} \langle  H_{\Psi} (w_k), w_k - w    \rangle.
\end{equation*}
Since $\nu > \frac{1}{p} >0$, and taking into account \eqref{ipotesi111}, we have
$$ \displaystyle \limsup_{k \to + \infty} \, \langle  H_{\Psi_{r,\gamma}}(w_k), w_k - w    \rangle  \leq 0.$$
By applying \cite[Lemma 2.1]{BCV2}, we infer that $w_k \to w$ in $\sobr$.
\end{proof}

\noindent
Now, let $\Psi_{1}, \Psi_{2}: \mathbb{R}^N \to \mathbb{R}$ be two functions as in assumption $(\Psi)$.\\
Let us consider the map $\mathcal{H}_{\Psi_1,\Psi_2}: X \to X'$ defined for any $z=(u,v) \in X$ by
\begin{align}\label{HPsi1Psi2}
\displaystyle
	 \mathcal{H}_{\Psi_1,\Psi_2}(z) = \left[ H_{\Psi_1}(u) , H_{\Psi_2}(v) \right], 
\end{align}
where the map $\mathcal{H}_{\Psi_1,\Psi_2}(z): X \to \mathbb{R}$ acts on any $\bar{z}=(\bar{u},\bar{v}) \in X$ as
\begin{align*}
\displaystyle
 \langle H_{\Psi_1,\Psi_2}(z), \bar{z} \rangle 
& =  \into \nabla \Psi_1 \left( \nabla u \right) \cdot \nabla \bar{u} \, dx + \into \nabla \Psi_2 \left( \nabla v \right) \cdot \nabla \bar{v} \, dx.
\end{align*}

\begin{proposition}\label{HPsi1Psi2S+}
The map $\mathcal{H}_{\Psi_1,\Psi_2}: X \to X'$ is of class $(S)_+$.
\end{proposition}

\begin{proof}
Let $\{z_k\}_k=\{(u_k,v_k)\}_k \subset X$ and $z=(u,v) \in X$ such that $z_k \rightharpoonup z$ in $X$ and
\begin{equation}\label{ipotesi11}
\displaystyle \limsup_{k \to + \infty} \, \langle  \mathcal{H}_{\Psi_1,\Psi_2}(z_k), z_k - z  \rangle \leq 0.
\end{equation}
Explicitly, we have
\begin{align*}
\displaystyle
 \langle \mathcal{H}_{\Psi_1,\Psi_2}(z_k), z_k - z \rangle 
& =  \into \nabla \Psi_1 \left( \nabla u_k \right) \cdot \nabla \left( u_k - u \right) \, dx \\
& +  \into \nabla \Psi_2 \left( \nabla v_k \right) \cdot \nabla \left( v_k - v \right) \, dx. \nonumber
\end{align*}
Since $u_k \rightharpoonup u $ in $W_0^{1,p}(\Omega)$ and $v_k \rightharpoonup v $ in $W_0^{1,q}(\Omega)$, we deduce 
\begin{equation}\label{conto1}
\displaystyle \into  \nabla \Psi_1 (\nabla u) \cdot \nabla \left(u_k - u \right) \, dx  \to 0,
\end{equation}
\begin{equation}\label{conto2}
\displaystyle \into  \nabla \Psi_{p,\alpha} (\nabla u) \cdot \nabla \left(u_k - u \right) \, dx  \to 0,
\end{equation}
and 
\begin{equation}\label{conto3}
\displaystyle \into  \nabla \Psi_2 (\nabla v) \cdot \nabla \left(v_k - v \right) \, dx  \to 0,
\end{equation}
\begin{equation}\label{conto4}
\displaystyle \into  \nabla \Psi_{q,\beta} (\nabla v) \cdot \nabla \left(v_k - v \right) \, dx  \to 0.
\end{equation}
By assumption $(\Psi)$, for any  $\xi_1, \xi_2 \in \mathbb{R}^N$ we have 
$$ \displaystyle \left( \nabla \Psi_1(\xi_2)  - \nabla \Psi_1(\xi_1) \right) \cdot \left( \xi_2 - \xi_1  \right) \geq \nu_1 \left( \nabla \Psi_{p,\alpha}  (\xi_2)  - \nabla \Psi_{p,\alpha}  (\xi_1) \right) \cdot \left( \xi_2 - \xi_1  \right) \geq 0 $$
and 
$$ \displaystyle \left( \nabla \Psi_2(\xi_2)  - \nabla \Psi_2(\xi_1) \right) \cdot \left( \xi_2 - \xi_1  \right) \geq \nu_2 \left( \nabla \Psi_{q,\beta}  (\xi_2)  - \nabla \Psi_{q,\beta}  (\xi_1) \right) \cdot \left( \xi_2 - \xi_1  \right) \geq 0. $$
Therefore, by using also the (strict) convexity of $\Psi_{p,\alpha}$ and $\Psi_{q,\beta},$ we get
\begin{align}\label{conto5}
\displaystyle 
      & \quad \into \left( \nabla \Psi_1  ( \nabla u_k)  - \nabla \Psi_1  (\nabla u) \right) \cdot \nabla \left(  u_k - u  \right) \, dx  \nonumber \\
    & + \into \left( \nabla \Psi_2  ( \nabla v_k)  - \nabla \Psi_2  (\nabla v) \right) \cdot \nabla \left(  v_k - v  \right) \, dx \nonumber \\
 \geq & \quad \nu_1 \into \left( \nabla \Psi_{p,\alpha}  ( \nabla u_k)  - \nabla \Psi_{p,\alpha}  (\nabla u) \right) \cdot \nabla \left(  u_k - u  \right) \, dx \\
     & + \nu_2 \into \left( \nabla \Psi_{q,\beta} ( \nabla v_k)  - \nabla \Psi_{q,\beta} (\nabla v) \right) \cdot \nabla \left(  v_k - v  \right) \, dx  \nonumber \\
 \geq & \min \{ \nu_1, \nu_2 \} \left( \into \left( \nabla \Psi_{p,\alpha}  ( \nabla u_k)  - \nabla \Psi_{p,\alpha}  (\nabla u) \right) \cdot \nabla \left(  u_k - u  \right) \, dx \right. \nonumber \\
    & \qquad \qquad \quad + \left. \into \left( \nabla \Psi_{q,\beta} ( \nabla v_k)  - \nabla \Psi_{q,\beta} (\nabla v) \right) \cdot \nabla \left(  v_k - v  \right) \, dx \right). \nonumber
\end{align}
Altogether, and considering also the properties of $\dis \limsup$, we infer that 
\begin{equation*}
\displaystyle  \min \{ \nu_1, \nu_2 \} \, \limsup_{k \to + \infty} \langle  \mathcal{H}_{\Psi_{p,\alpha},\Psi_{q,\beta}}(z_k), z_k - z    \rangle \leq \limsup_{k \to + \infty} \langle  \mathcal{H}_{\Psi_1,\Psi_2} (z_k), z_k - z    \rangle.
\end{equation*}
Since $\min \{ \nu_1, \nu_2 \}  >0$, and taking into account \eqref{ipotesi11},  we proved
$$ \displaystyle \limsup_{k \to + \infty} \langle  \mathcal{H}_{\Psi_{p,\alpha},\Psi_{q,\beta}}(z_k), z_k - z    \rangle \leq 0. $$
By applying \cite[Proposition 2.2.]{BCV2}, we have the thesis.
\end{proof}

\begin{remark}\label{dausarenellaProof}
Considering in the previous proof the limits \eqref{conto1}, \eqref{conto2}, \eqref{conto3}, \eqref{conto4} and the inequality \eqref{conto5}, and taking into account a sequence $\{A_k\}_k$  of real numbers, we get 
\begin{align*}
\dis 
& \limsup_{k \to + \infty} \, \biggl( \nu_1 \,  \langle H_{\Psi_{p,\alpha}}(u_k),u_k-u  \rangle + \nu_2 \, \langle H_{\Psi_{q,\beta}}(v_k),v_k-v    \rangle    - A_k  \biggr) \\
& \qquad \qquad \qquad \leq \limsup_{k \to + \infty} \, \biggl(  \langle \mathcal{H}_{\Psi_1,\Psi_2}(z_k), z_k - z    \rangle    - A_k  \biggr).
\end{align*}
\end{remark}

\begin{remark}\label{TroncataRresto}
For every $h \in \mathbb{N},$ $h \geq 1$ we define $\mathcal{T}_h,\mathcal{R}_h:\mathbb{R} \longrightarrow \mathbb{R}$ as 
\begin{equation}\label{ThRh}
\displaystyle \mathcal{T}_h(s):= \min\{\max\{s,-h\},h\} \qquad \text{and}  \qquad \mathcal{R}_h(s):=s-\mathcal{T}_h(s).
\end{equation}
If $w \in W_0^{1,r}(\Omega)$ for some $r>1$,  then $\mathcal{T}_h(w), \, \mathcal{R}_h(w) \in W_0^{1,r}(\Omega)$. Moreover, we have
\begin{equation*}
\qquad \mathcal{T}_h(w(x))=
\begin{cases}
   h       & \text{ a.e. in } \{ x \in \Omega \; : \; w(x)\geq  h \}, \medskip \\
    w(x)             & \text{ a.e. in } \{ x \in \Omega \; : \; - h < w(x) < h \}, \medskip \\
 - h        & \text{ a.e. in } \{ x \in \Omega  \; : \; w(x) \leq - h \};
\end{cases}
\end{equation*}

\begin{equation*}
 \nabla \mathcal{T}_h(w(x))=
\begin{cases}
 \nabla w(x)        & \text{ a.e. in } \{ x \in \Omega \;  : \; |w(x)| <  h \}, \medskip\\
0                   & \text{ a.e. in } \{ x \in \Omega \;  : \; |w(x)| \geq h \};
\end{cases}
\end{equation*}

\begin{equation*}
\qquad \qquad \mathcal{R}_h(w(x))=
\begin{cases}
w(x) -  h       & \text{ a.e. in } \{ x \in \Omega \; : \; w(x)\geq  h \}, \medskip \\
0               & \text{ a.e. in } \{ x \in \Omega \; : \; - h < w(x) < h \}, \medskip \\
w(x) + h        & \text{ a.e. in } \{ x \in \Omega  \; : \; w(x) \leq - h \};
\end{cases}
\end{equation*}

\begin{equation*}
 \nabla \mathcal{R}_h(w(x))=
\begin{cases}
 \nabla w(x)        & \text{ a.e. in } \{ x \in \Omega \;  : \; |w(x)| >  h \}, \medskip\\
0                   & \text{ a.e. in } \{ x \in \Omega \;  : \; |w(x)| \leq h \};
\end{cases}
\end{equation*}

\medskip
\noindent
and 

\begin{equation*}
\qquad \qquad \nabla w(x)=
\begin{cases}
 \nabla \mathcal{T}_h(w(x))        & \text{ a.e. in } \{ x \in \Omega \; : \; |w(x)|< h \}, \medskip \\
 \nabla \mathcal{R}_h(w(x))                  & \text{ a.e. in } \{ x \in \Omega \; : \; |w(x)| > h \}, \medskip \\
 0       & \text{ a.e. in } \{ x \in \Omega  \; : \; |w(x)|=h \}.
\end{cases}
\end{equation*}
\end{remark}

\medskip
\noindent
We will use the following result, proved in \cite[Appendix]{BCV2} .

\begin{lemma}\label{convergenzasopralivelli}
 Let $\Omega$ be a bounded domain of $\mathbb{R}^N$.
Let $\{w_k\}_k \subset L^1(\Omega)$ and $w \in L^1(\Omega)$ such that $w_k(x) \to w(x)$ a.e. $x \in \Omega$ as $k \to \infty$. Let $h \in \mathbb{N}$, $h \geq 1$ and let us define, almost everywhere, the sets
\begin{equation*}
\displaystyle
\Omega^{w_k}_{h,+}:= \{ x \in \Omega \; : \;  |w_k(x)| \;  \geq \;  h   \} \quad \text{and} \quad \Omega^{w}_{h,+}:= \{ x \in \Omega \; : \;  |w(x)| \;  \geq \;  h   \}.
\end{equation*} 
Then 
$$ \displaystyle \limsup_{k \to + \infty} \left| \Omega^{w_k}_{h,+}  \right| \leq \left| \Omega^{w}_{h,+}  \right|.$$ 
\end{lemma}

\medskip
\noindent
We will frequently use also the following result (see \cite{BREZIS}).

\begin{proposition}\label{weakconvLp}
Let $1< r < \infty$ and $(\Omega,\mathcal{M}, \mu)$ a measure space with $\sigma$-finite measure $\mu$.
Let $\{f_k\}_k \subset L^r(\Omega)$ be a sequence such that:
\begin{itemize}
\item[i)] $\{f_k\}_k$ is bounded in $L^r(\Omega)$;
\item[ii)] $f_k(x) \to f(x)$ a.e. $x \in \Omega$. 
\end{itemize}
Then 
$$f_k \rightharpoonup f \text{ in } L^r(\Omega).$$
\end{proposition}

\bigskip
\section{PROOF OF THEOREM ~\ref{IprimodeltaPsi1Psi2S+}.~}

\medskip
\noindent
Let us consider $z_0=(u_0,v_0) \in X$, a radius $R>0$, and let us consider a sequence $\{z_k\}_k=\{(u_k,v_k)\}_k \subset \overline{B_R(z_0)}$  weakly convergent to some $z=(u,v)$ in $X$ with
\begin{equation*}
\displaystyle \limsup_{k\to + \infty} \, \langle I'_{\delta,\Psi_1,\Psi_2}(z_k), z_k-z\rangle\leq 0.
\end{equation*} 
Via the map $\mathcal{H}_{\Psi_1,\Psi_2}$ defined in \eqref{HPsi1Psi2}, the preceding assumption can be rewritten as follows:
\begin{align}\label{ipotesiHapbq}
\displaystyle 
0 \geq \limsup_{k \to + \infty}
& \Biggl\{  \langle \mathcal{H}_{\Psi_1,\Psi_2}(z_k),z_k-z    \rangle  \\ \nonumber
&   - \into \left[ H_s(\delta,x,u_k,v_k) (u_k - u) + H_t(\delta,x,u_k,v_k)(v_k - v) \right]
dx  \Biggr\} .
\end{align}
By Remark \ref{dausarenellaProof}, it follows that
\begin{align*}
\displaystyle 
0 \geq \limsup_{k \to + \infty}
&  \Biggl\{ \nu_1 \,  \langle H_{\Psi_{p,\alpha}}(u_k),u_k-u  \rangle + \nu_2 \, \langle H_{\Psi_{q,\beta}}(v_k),v_k-v    \rangle \\ \nonumber
&  - \into \left[ H_s(\delta,x,u_k,v_k) (u_k - u) + H_t(\delta,x,u_k,v_k)(v_k - v) \right]
dx   \Biggr\}. 
\end{align*}

Taking into account $c \geq 0, r >1$ and the convexity of the functional $f:W_0^{1,r}(\Omega) \to \mathbb{R}$ defined as
\begin{equation}\label{funzionalegenerico}
\displaystyle f(w):= \frac{1}{r} \into \left( c^2 + | \nabla w |^2 \right)^{\frac{r}{2}} \, dx.
\end{equation} 
we get
\begin{align}\label{contoPSlocale1}
\dis
0 & \geq \limsup_{k \to + \infty} \Biggl\{  \frac{\nu_1}{p} \into \left( \alpha^2 + | \nabla u_k|^2 \right)^{\frac{p}{2}} \, dx  - \frac{\nu_1}{p} \into \left( \alpha^2 + | \nabla u |^2 \right)^{\frac{p}{2}} \, dx  \\
&   \qquad \qquad + \frac{\nu_2}{q} \into \left( \beta^2 + | \nabla v_k|^2 \right)^{\frac{q}{2}} \, dx  - \frac{\nu_2}{q} \into \left( \beta^2 + | \nabla v |^2 \right)^{\frac{q}{2}} \, dx \nonumber \\ 
&   \qquad \qquad- \into \left[ H_s(\delta,x,u_k,v_k) (u_k - u) + H_t(\delta,x,u_k,v_k)(v_k - v) \right]
 \, dx \Biggr\} . \nonumber 
\end{align}
We know that $u_k \rightharpoonup u$ in $W_0^{1,p}(\Omega)$ and $v_k \rightharpoonup v$ in $W_0^{1,q}(\Omega)$.  Moreover, we will use the fact that  $\{u_k\}_k$ is bounded in $L^{p^*}(\Omega)$ and
 $\{v_k\}_k$ is bounded in $L^{q^*}(\Omega)$. For every $h \in \mathbb{N}$, $h \geq 1$, let us consider the functions $\mathcal{T}_h,\mathcal{R}_h:\mathbb{R} \to \mathbb{R}$ defined in \eqref{ThRh}. We have:\\
\begin{itemize}[leftmargin=*]
\item[$i)$] $\mathcal{T}_h(u_k) \rightharpoonup \mathcal{T}_h(u) \text{ in } W_0^{1,p}(\Omega)$, and $\mathcal{T}_h(v_k) \rightharpoonup \mathcal{T}_h(v) \text{ in } W_0^{1,q}(\Omega);$ 
\item[$ii)$] $\mathcal{T}_h(u_k) \rightarrow \mathcal{T}_h(u) \text{ in }  L^{\sigma}(\Omega)$ if $ 1 \leq \sigma < \infty$, and $\mathcal{T}_h(v_k) \rightarrow \mathcal{T}_h(v) \text{ in }  L^{\tau}(\Omega)$ if $ 1 \leq \tau < \infty$. 
\item[$iii)$] $\mathcal{R}_h(u_k) \rightharpoonup \mathcal{R}_h(u) \text{ in } W_0^{1,p}(\Omega)$, and $\mathcal{R}_h(v_k) \rightharpoonup \mathcal{R}_h(v) \text{ in } W_0^{1,q}(\Omega);$
\item[$iv)$] $\mathcal{R}_h(u_k) \rightarrow \mathcal{R}_h(u) \text{ in }  L^{\sigma}(\Omega)$ if $ 1 \leq \sigma < p^*$, and $\mathcal{R}_h(v_k) \rightarrow \mathcal{R}_h(v) \text{ in }  L^{\tau}(\Omega)$ if $ 1 \leq \tau < q^*$.
\end{itemize}
Taking into account Remark \ref{TroncataRresto}, we have
\begin{align*}
\displaystyle
 &\into \left( \alpha^2 + | \nabla u|^2 \right)^{\frac{p}{2}} \, dx \\
=&  \int_{ |u| > h  }  \left( \alpha^2 + | \nabla R_h (u) |^2 \right)^{\frac{p}{2}} \, dx 
 + \int_{ |u| < h  }  \left( \alpha^2 + | \nabla T_h (u) |^2 \right)^{\frac{p}{2}} \, dx 
 + \alpha^{p}  \int_{ |u| = h  } \, dx \\
= &  \quad  \into \left( \alpha^2 + | \nabla R_h(u)|^2 \right)^{\frac{p}{2}} \, dx - \alpha^{\frac{p}{2}}  \int_{ |u| \leq h  } \, dx \\
& + \into  \left( \alpha^2 + | \nabla T_h (u) |^2 \right)^{\frac{p}{2}} \, dx - \alpha^{p}  \int_{ |u| \geq h  } \, dx + \alpha^{p}  \int_{ |u| = h  } \, dx \\
=&  \into \left( \alpha^2 + | \nabla R_h(u)|^2 \right)^{\frac{p}{2}} \, dx  + \into \left( \alpha^2 + | \nabla T_h(u)|^2 \right)^{\frac{p}{2}} \, dx  - \alpha^{p} \left| \Omega   \right|, 
\end{align*}
Similarly, we have
\begin{align*}
\displaystyle
& \into \left( \alpha^2 + | \nabla u_k|^2 \right)^{\frac{p}{2}} \, dx \\
= &  \into \left( \alpha^2 + | \nabla R_h(u_k)|^2 \right)^{\frac{p}{2}} \, dx 
+ \into \left( \alpha^2 + | \nabla T_h(u_k)|^2 \right)^{\frac{p}{2}} \, dx  - \alpha^{p} \left| \Omega   \right|,\\
\end{align*}
\begin{align*}
\displaystyle
 & \into \left( \beta^2+ | \nabla v_k|^2 \right)^{\frac{q}{2}} \, dx \\
= & \into \left( \beta^2 + | \nabla R_h(v_k)|^2 \right)^{\frac{q}{2}} \, dx 
+ \into \left( \beta^2 + | \nabla T_h(v_k)|^2 \right)^{\frac{q}{2}} \, dx  - \beta^{q} \left| \Omega   \right|,\\
\end{align*}
\begin{align*}
\displaystyle
 &\into \left( \beta^2 + | \nabla v|^2 \right)^{\frac{q}{2}} \, dx \\
= & \into \left( \beta^2 + | \nabla R_h(v)|^2 \right)^{\frac{q}{2}} \, dx 
+ \into \left( \beta^2 + | \nabla T_h(v)|^2 \right)^{\frac{q}{2}} \, dx  - \beta^{q} \left| \Omega   \right|.
\end{align*}
Considering these relations in \eqref{contoPSlocale1}, we get 
\begin{align}\label{contoPSlocale2}
\displaystyle 
0  & \geq  \limsup_{k \to + \infty}  \;  \left\lbrace  \frac{\nu_1}{p} \into \left( \alpha^2 + | \nabla \mathcal{R}_h(u_k)|^2 \right)^{\frac{p}{2}} \, dx - \frac{\nu_1}{p} \into \left( \alpha^2 + | \nabla \mathcal{R}_h(u)|^2 \right)^{\frac{p}{2}} \, dx \right.  \\
& \quad \qquad \quad + \frac{\nu_2}{q} \into \left( \beta^2 + | \nabla \mathcal{R}_h(v_k)|^2 \right)^{\frac{q}{2}} \, dx - \frac{\nu_2}{q} \into \left( \beta^2 + | \nabla \mathcal{R}_h(v)|^2 \right)^{\frac{q}{2}} \, dx \nonumber \\
& \quad \qquad \quad - \into  H_s(\delta,x,u_k,v_k) (\mathcal{R}_h(u_k) - \mathcal{R}_h(u))  \, dx \nonumber \\
& \quad \qquad \quad  - \into  H_s(\delta,x,u_k,v_k)(\mathcal{T}_h(u_k) - \mathcal{T}_h(u)) \, dx \nonumber \\
& \quad \qquad \quad - \into  H_t(\delta,x,u_k,v_k) (\mathcal{R}_h(v_k) - \mathcal{R}_h(v))  \, dx \nonumber \\
& \quad \qquad \quad  - \into  H_t(\delta,x,u_k,v_k)(\mathcal{T}_h(v_k) - \mathcal{T}_h(v)) \, dx \nonumber \\
& \quad \qquad \quad + \frac{\nu_1}{p} \into \left( \alpha^2 + | \nabla \mathcal{T}_h(u_k)|^2 \right)^{\frac{p}{2}} \, dx - \frac{\nu_1}{p} \into \left( \alpha^2 + | \nabla \mathcal{T}_h(u)|^2 \right)^{\frac{p}{2}} \, dx \nonumber \\
& \left. \qquad \quad \quad + \frac{\nu_2}{q} \into \left( \beta^2 + | \nabla \mathcal{T}_h(v_k)|^2 \right)^{\frac{q}{2}} \, dx - \frac{\nu_2}{q} \into \left( \beta^2 + | \nabla \mathcal{T}_h(v)|^2 \right)^{\frac{q}{2}} \, dx \right\rbrace.  \nonumber
\end{align}
Considering $ii)$, assumption $(H)$ and the boundedness of $\{u_k\}_k $ in $ L^{p^*}(\Omega)$ and of $\{v_k\}_k$ in $L^{q^*}(\Omega)$, we get
\begin{align*}
\dis \into  H_s(\delta,x,u_k,v_k) (\mathcal{T}_h(u_k) - \mathcal{T}_h(u)) \, dx \to 0
\end{align*}
and
\begin{align*}
\dis \into  H_t(\delta,x,u_k,v_k)  (\mathcal{T}_h(v_k) - \mathcal{T}_h(v)) \, dx \to 0. 
\end{align*}
Moreover, by $i)$ and the weak lower semicontinuity of the functional $f$ defined in \eqref{funzionalegenerico}, we get
\begin{align*}
\frac{\nu_1}{p} \into \left( \alpha^2 + | \nabla \mathcal{T}_h(u) |^2 \right)^{\frac{p}{2}} \, dx  \leq \liminf_{k \to + \infty} \left[ \frac{\nu_1}{p} \into \left( \alpha^2 + | \nabla \mathcal{T}_h(u_k) |^2 \right)^{\frac{p}{2}} \, dx \right]
\end{align*}
and
\begin{align*}
 \frac{\nu_2}{q} \into \left( \beta^2 + | \nabla \mathcal{T}_h(v) |^2 \right)^{\frac{q}{2}} \, dx  \leq \liminf_{k \to + \infty} \left[ \frac{\nu_2}{q} \into \left( \beta^2 + | \nabla \mathcal{T}_h(v_k) |^2 \right)^{\frac{q}{2}} \, dx \right] . 
\end{align*}
Considering also the properties of $\limsup$ and $\liminf$, 
by \eqref{contoPSlocale2}  we deduce 
\begin{align*}
\displaystyle 
0  & \geq  \limsup_{k \to + \infty}  \left\lbrace \quad  \frac{\nu_1}{p} \into \left( \alpha^2 + | \nabla \mathcal{R}_h(u_k)|^2 \right)^{\frac{p}{2}} \, dx - \frac{\nu_1}{p} \into \left( \alpha^2 + | \nabla \mathcal{R}_h(u)|^2 \right)^{\frac{p}{2}} \, dx \right. \\
& \quad \qquad \qquad + \frac{\nu_2}{q} \into \left( \beta^2 + | \nabla \mathcal{R}_h(v_k)|^2 \right)^{\frac{q}{2}} \, dx - \frac{\nu_2}{q} \into \left( \beta^2 + | \nabla \mathcal{R}_h(v)|^2 \right)^{\frac{q}{2}} \, dx \\
&       \quad \qquad \qquad - \into  H_s(\delta,x,u_k,v_k)  (\mathcal{R}_h(u_k) - \mathcal{R}_h(u)) \, dx \\
& \left. \quad \qquad \qquad - \into  H_t(\delta,x,u_k,v_k) (\mathcal{R}_h(v_k) - \mathcal{R}_h(v))  \, dx \quad \right\rbrace . 
\end{align*}
Let us now define, almost everywhere, the following sets:
\begin{align*}
\displaystyle
\Omega^{u_k}_{h,+}:= \{ x \in \Omega \; : \;  |u_k(x)| \;  \geq  \;  h   \} \quad \text{and} \quad \Omega^{u_k}_{h,-}:= \{ x \in \Omega \; : \;  |u_k(x)| \;   <     \;  h   \}. 
\end{align*} 
Taking into account  Remark \ref{TroncataRresto}, we observe that 
\begin{align*}
\displaystyle
   \into \left( \alpha^2 + | \nabla \mathcal{R}_h(u_k)|^2 \right)^{\frac{p}{2}} \, dx 
& =    \alpha^{p} \left| \Omega^{u_k}_{h,-}  \right|   +  \int_{ \Omega^{u_k}_{h,+}} \left( \alpha^2 + | \nabla \mathcal{R}_h(u_k)|^2 \right)^{\frac{p}{2}} \, dx \\
& \geq  \alpha^{p} \left| \Omega^{u_k}_{h,-}  \right|   +  \int_{ \Omega^{u_k}_{h,+}} | \nabla \mathcal{R}_h(u_k)|^p \, dx \\
& = \alpha^{p} \left| \Omega \right| +  \into | \nabla \mathcal{R}_h(u_k)|^p \, dx - \alpha^{p} \left| \Omega^{u_k}_{h,+}  \right|. 
\end{align*}
By Lemma \ref{convergenzasopralivelli}, we have  
$$ \displaystyle \limsup_{k \to + \infty} \left| \Omega^{u_k}_{h,+}  \right| \leq \left| \Omega^{u}_{h,+}  \right|.$$
Similarly, we have 
\begin{align*}
\displaystyle
   \into \left( \beta^2 + | \nabla \mathcal{R}_h(v_k)|^2 \right)^{\frac{q}{2}} \, dx  \geq \beta^{q} \left| \Omega \right| +  \into | \nabla \mathcal{R}_h(v_k)|^q \, dx - \beta^{q} \left| \Omega^{v_k}_{h,+}  \right|
\end{align*}
with
$$ \displaystyle \limsup_{k \to + \infty} \left| \Omega^{v_k}_{h,+}  \right| \leq \left| \Omega^{v}_{h,+}  \right|.$$
Hence, considering once again the properties of $\limsup$ and $\liminf$, 
so far we have proved
\begin{align*}
\displaystyle
\limsup_{k \to + \infty}  & \Biggl\{ \frac{\nu_1}{p} \into | \nabla \mathcal{R}_h(u_k)|^p \, dx  +  \frac{\nu_2}{q} \into | \nabla \mathcal{R}_h(v_k)|^q \, dx  \\
& \quad - \into H_s(\delta,x,u_k,v_k) \left( \mathcal{R}_h(u_k)-\mathcal{R}_h(u) \right) \, dx \\
& \quad - \into H_t(\delta,x,u_k,v_k) \left( \mathcal{R}_h(v_k)-\mathcal{R}_h(v) \right)  \, dx  \Biggr\} \\
& \leq \quad \frac{\nu_1}{p} \into \left( \alpha^2+ | \nabla \mathcal{R}_h(u)|^2 \right)^{\frac{p}{2}} \, dx - \frac{\nu_1}{p} \alpha^{p} \left| \Omega \right| + \frac{\nu_1}{p} \alpha^{p} \left| \Omega^{u}_{h,+}  \right| \\
& \quad + \frac{\nu_2}{q} \into \left( \beta^2 + | \nabla \mathcal{R}_h(v)|^2 \right)^{\frac{q}{2}} \, dx -  \frac{\nu_2}{q} \beta^{q} \left| \Omega \right| + \frac{\nu_2}{q} \beta^{q} \left| \Omega^{v}_{h,+}  \right| . 
\end{align*}
Now, on the one hand by Chebyshev inequality as $h \to  \infty$  we have 
$$ \displaystyle \left| \Omega^{u}_{h,+}  \right| \to 0 \quad \text{and} \quad \left| \Omega^{v}_{h,+}  \right| \to 0, $$
on the other hand by dominated convergence theorem as $h \to  \infty$ we have
$$ \displaystyle \into \left( \alpha^2+ | \nabla \mathcal{R}_h(u)|^2 \right)^{\frac{p}{2}} \, dx \to \alpha^{p} \left| \Omega \right|  $$
and
$$ \displaystyle  \into \left( \beta^2+ | \nabla \mathcal{R}_h(v)|^2 \right)^{\frac{q}{2}} \, dx \to \beta^{q} \left| \Omega \right|. $$
Therefore, denoting by $o_h(1)$ a quantity that goes to $0$ as $h \to \infty$, we obtained
\begin{align*}
\displaystyle
\limsup_{k \to + \infty} &  \Biggl\{ \frac{\nu_1}{p} \into | \nabla \mathcal{R}_h(u_k)|^p \, dx  +  \frac{\nu_2}{q} \into | \nabla \mathcal{R}_h(v_k)|^q \, dx   \\
&   - \into H_s(\delta,x,u_k,v_k) \left( \mathcal{R}_h(u_k)-\mathcal{R}_h(u) \right) \, dx \\
&   - \into H_t(\delta,x,u_k,v_k) \left( \mathcal{R}_h(u_k)-\mathcal{R}_h(u) \right) \, dx \Biggr\} \\
& \leq o_h(1).
\end{align*}
By continuity of $H_s(\delta,x,\cdot,\cdot)$ and the convergence a.e. $x \in \Omega$ of $u_k(x)$ to $u(x)$ and of $v_k(x)$ to $v(x)$ as $k \to \infty$, we deduce
$$H_s( \delta,x,u_k(x),v_k(x)) \to H_s(\delta,x,u(x),v(x)) \qquad \text{a.e. } x \in \Omega \qquad \text{as } k \to \infty.$$
By $(H)$ and the boundedness of $\{u_k\}_k$ in $L^{p^*}(\Omega)$ and of $\{v_k\}_k$ in $L^{q^*}(\Omega)$, we deduce that the sequence $\{H_s(\delta,\cdot,u_k(\cdot),v_k(\cdot))\}_k$ is bounded in $L^{(p^*)'}(\Omega)$. Hence, we can apply Proposition \ref{weakconvLp} and say that 
$$H_s(\delta, \cdot, u_k(\cdot),v_k(\cdot)) \rightharpoonup H_s(\delta, \cdot, u(\cdot),v(\cdot)) \text{ in } L^{(p^*)'}(\Omega).$$ 
Similarly, also 
$$H_t(\delta, \cdot, u_k(\cdot),v_k(\cdot)) \rightharpoonup H_t(\delta, \cdot, u(\cdot),v(\cdot)) \text{ in } L^{(q^*)'}(\Omega).$$
Since $\mathcal{R}_h(u) \in L^{p^*}(\Omega)$ and $\mathcal{R}_h(v) \in L^{q^*}(\Omega)$,  we infer that
$$ \dis  \into H_s(\delta,x,u_k,v_k)\mathcal{R}_h(u) \, dx \to \into H_s(\delta,x,u,v)\mathcal{R}_h(u) \, dx$$
and
$$ \dis \into H_t(\delta,x,u_k,v_k)\mathcal{R}_h(v) \, dx \to \into H_t(\delta,x,u,v)\mathcal{R}_h(v) \, dx,$$
whence
\begin{align*}
\displaystyle
\limsup_{k \to + \infty} &  \Biggl\{ \quad \frac{\nu_1}{p} \into | \nabla \mathcal{R}_h(u_k)|^p \, dx  +  \frac{\nu_2}{q} \into | \nabla \mathcal{R}_h(v_k)|^q \, dx   \\
&   \quad - \into H_s(\delta,x,u_k,v_k) \mathcal{R}_h(u_k) \, dx \\
&   \quad - \into H_t(\delta,x,u_k,v_k) \mathcal{R}_h(u_k) \, dx \Biggr\} \\
& \leq - \into H_s(\delta,x,u,v) \mathcal{R}_h(u) \, dx \\
& \quad - \into H_t(\delta,x,u,v) \mathcal{R}_h(v) \, dx + o_h(1).
\end{align*}
By  $iv)$ we deduce $ \into \left| \mathcal{R}_h(u_k)  \right| \, dx \to \into \left|  \mathcal{R}_h(u)  \right| \, dx $  and $ \into \left| \mathcal{R}_h(v_k)  \right| \, dx \to \into \left|  \mathcal{R}_h(v)  \right| \, dx $ as $ k \to \infty.$ Together with assumption $(H)$, we get 
\begin{align*}
\displaystyle
\limsup_{k \to + \infty} 
& \Biggl\{ \quad \frac{\nu_1}{p} \into \left| \nabla \mathcal{R}_h(u_k)  \right|^p \, dx +  \frac{\nu_2}{q} \into \left| \nabla \mathcal{R}_h(v_k)  \right|^q \, dx   \\
& \quad - C \into |u_k|^{p^*-1} \left| \mathcal{R}_h(u_k) \right| \, dx - C \into |v_k|^{q^{*}\frac{p^*-1}{p^{*}}} \left| \mathcal{R}_h(u_k) \right| \, dx  \\
&  \quad - C \into |u_k|^{p^*\frac{q^{*}-1}{q^{*}}} \left| \mathcal{R}_h(v_k) \right| \, dx - C \into |v_k|^{q^*-1} \left| \mathcal{R}_h(v_k) \right| \, dx  \Biggr\} \\
& \leq  \quad C \into |u|^{p^*-1} \left| \mathcal{R}_h(u) \right| \, dx + C \into |v|^{q^{*}\frac{p^*-1}{p^{*}}} \left| \mathcal{R}_h(u) \right| \, dx  \\
& \quad + C \into |u|^{p^*\frac{q^{*}-1}{q^{*}}} \left| \mathcal{R}_h(v) \right| \, dx + C \into |v|^{q^*-1} \left| \mathcal{R}_h(v) \right| \, dx \\
& \quad + 2C \into \left| \mathcal{R}_h(u) \right| \, dx  + 2C \into \left| \mathcal{R}_h(v) \right| \, dx + o_h(1).
\end{align*}
By Gagliardo–Nirenberg–Sobolev inequality there exist $S(p,N)>0$ and $S(q,N)>0$ such that
$$ \displaystyle S(p,N) \lVert \bar{u} \rVert_{p^*}^{p} \leq \into | \nabla \bar{u}|^p \, dx \qquad \forall \, \bar{u} \in W_0^{1,p}(\Omega) $$
and
$$   S(q,N) \lVert \bar{v} \rVert_{q^*}^{q} \leq \into | \nabla \bar{v}|^q \, dx \qquad \forall \, \bar{v} \in W_0^{1,q}(\Omega).$$
Therefore,
\begin{align}\label{calcoloPalaisSmale3}
\displaystyle
\limsup_{k \to + \infty} 
& \Biggl\{ \quad \nu_1 \frac{ S(p,N)}{p} \lVert \mathcal{R}_h(u_k) \rVert_{p^*}^{p} +  \nu_2 \frac{ S(q,N)}{q} \lVert \mathcal{R}_h(v_k) \rVert_{q^*}^{q}  \\
& \quad - C \into |u_k|^{p^*-1} \left| \mathcal{R}_h(u_k) \right| \, dx - C \into |v_k|^{q^{*}\frac{p^*-1}{p^{*}}} \left| \mathcal{R}_h(u_k) \right| \, dx  \nonumber \\
&  \quad - C \into |u_k|^{p^*\frac{q^{*}-1}{q^{*}}} \left| \mathcal{R}_h(v_k) \right| \, dx - C \into |v_k|^{q^*-1} \left| \mathcal{R}_h(v_k) \right| \, dx  \Biggr\} \nonumber \\
& \leq  \quad C \into |u|^{p^*-1} \left| \mathcal{R}_h(u) \right| \, dx + C \into |v|^{q^{*}\frac{p^*-1}{p^{*}}} \left| \mathcal{R}_h(u) \right| \, dx \nonumber  \\
& \quad + C \into |u|^{p^*\frac{q^{*}-1}{q^{*}}} \left| \mathcal{R}_h(v) \right| \, dx + C \into |v|^{q^*-1} \left| \mathcal{R}_h(v) \right| \, dx \nonumber  \\
& \quad + 2C \into \left| \mathcal{R}_h(u) \right| \, dx  + 2C \into \left| \mathcal{R}_h(v) \right| \, dx + o_h(1). \nonumber
\end{align}
Let us now consider the term
$$ \into |u_k|^{p^*-1} \left| \mathcal{R}_h(u_k) \right| \, dx.$$
By H\"older inequality with  conjugated exponents $p^*/p$ and $p^*/(p^* - p)$, we have
\begin{align*}
\displaystyle
 &\into |u_k|^{p^*-1} \left| \mathcal{R}_h(u_k) \right| \, dx \\
= &  \into |u_k|^{p^*-p} \left| \mathcal{T}_h(u_k) + \mathcal{R}_h(u_k)  \right|^{p-1} \left| \mathcal{R}_h(u_k) \right| \, dx \\
\leq & c_p \left( \into |u_k|^{p^*-p} \left| \mathcal{T}_h(u_k)  \right|^{p-1} \left| \mathcal{R}_h(u_k) \right| \, dx + 
\into |u_k|^{p^*-p} \left| \mathcal{R}_h(u_k)  \right|^{p}  \, dx \right)  \\
 \leq & c_p \left( \into |u_k|^{p^*-p} \left| \mathcal{T}_h(u_k)  \right|^{p-1} \left| \mathcal{R}_h(u_k) \right| \, dx +  \into |u_0|^{p^*-p} \left| \mathcal{R}_h(u_k)  \right|^{p}  \, dx \right) \\
& \quad +  c_p \into |u_k - u_0|^{p^*-p} \left| \mathcal{R}_h(v_k)  \right|^{p}  \, dx \\
 \leq &  c_p \left( \into |u_k|^{p^*-p} \left| \mathcal{T}_h(u_k)  \right|^{p-1} \left| \mathcal{R}_h(u_k) \right| \, dx  + \into |u_0|^{p^*-p} \left| \mathcal{R}_h(u_k)  \right|^{p}  \, dx \right) \\
& \quad  +  c_p \lVert u_k - u_0 \rVert_{p^*}^{p^* - p} \lVert \mathcal{R}_h(u_k) \rVert_{p^*}^p.
\end{align*}
By Proposition \ref{weakconvLp}, we infer that $ \left| \mathcal{R}_h(u_k)  \right|^{p} \rightharpoonup \left| \mathcal{R}_h(u)  \right|^{p}$ in $L^{\frac{p^*}{p}}(\Omega)$. Since $|u_0|^{p^*-p} \in L^{ \frac{p^*}{p^*-p}}(\Omega)$, we have
$$ \displaystyle \into |u_0|^{p^*-p} \left| \mathcal{R}_h(u_k)  \right|^{p}  \, dx \to \into |u_0|^{p^*-p} \left| \mathcal{R}_h(u)  \right|^{p}  \, dx. $$
On the other hand, by dominated convergence theorem we have 
$$ \displaystyle \into |u_k|^{p^*-p} \left| \mathcal{T}_h(u_k)  \right|^{p-1} \left| \mathcal{R}_h(u_k) \right| \, dx \to \into |u|^{p^*-p} \left| \mathcal{T}_h(u)  \right|^{p-1} \left| \mathcal{R}_h(u) \right| \, dx.$$
Summing up, and denoting by $o_k(1)$ a quantity the goes to zero as $k \to \infty$, we have shown that 
\begin{align*}
\displaystyle
 & \into |u_k|^{p^*-1} \left| \mathcal{R}_h(u_k) \right| \, dx \\
\leq  & c_p \left( \into |u|^{p^*-p} \left| \mathcal{T}_h(v)  \right|^{q-1} \left| \mathcal{R}_h(u) \right| \, dx  + \into |u_0|^{p^*-p} \left| \mathcal{R}_h(u)  \right|^{p}  \, dx \right) \\
& + c_p \lVert u_k - u_0 \rVert_{p^*}^{p^* - p} \lVert \mathcal{R}_h(u_k) \rVert_{p^*}^p + o(1).
\end{align*}
Similarly, we have
\begin{align*}
\displaystyle
 & \into |v_k|^{q^*-1} \left| \mathcal{R}_h(v_k) \right| \, dx  \\
\leq  & c_q \left( \into |v|^{q^*-q} \left| \mathcal{T}_h(v)  \right|^{q-1} \left| \mathcal{R}_h(v) \right| \, dx  + \into |v_0|^{q^*-q} \left| \mathcal{R}_h(v)  \right|^{q}  \, dx \right) \\
& + c_q \lVert v_k - v_0 \rVert_{q^*}^{q^* - q} \lVert \mathcal{R}_h(v_k) \rVert_{q^*}^q + o(1).
\end{align*}
Let us now consider the term
$$ \displaystyle \into  |u_k|^{p^*\frac{q^{*}-1}{q^{*}}} \left| \mathcal{R}_h(v_k) \right| \, dx. $$
Let us denote by
$$ \displaystyle \gamma:= p^*\frac{q^{*}-1}{q^{*}} - \frac{p}{q'}=\frac{p^* \, q^* - p^* -  \frac{p}{q'}q^*}{q^*}, $$
We have
\begin{align*}
\displaystyle
& \into  |u_k|^{p^*\frac{q^{*}-1}{q^{*}}} \left| \mathcal{R}_h(v_k) \right| \, dx \\
= & \into  |u_k|^{\gamma} \left| \mathcal{T}_h(u_k) + \mathcal{R}_h(u_k) \right|^{\frac{p}{q'}} \left| \mathcal{R}_h(v_k) \right| \, dx \\
 \leq & c_{p,q} \left( \into  |u_k|^{\gamma} \left| \mathcal{T}_h(u_k) \right|^{\frac{p}{q'}} \left| \mathcal{R}_h(v_k) \right| \, dx  + \into  |u_k|^{\gamma} \left| \mathcal{R}_h(u_k) \right|^{\frac{p}{q'}} \left| \mathcal{R}_h(v_k) \right| \, dx      \right) \\
\leq &   c_{p,q} \left( \into  |u_k|^{\gamma} \left| \mathcal{T}_h(u_k) \right|^{\frac{p}{q'}} \left| \mathcal{R}_h(v_k) \right| \, dx  +  \into  |u_0|^{\gamma} \left| \mathcal{R}_h(u_k) \right|^{\frac{p}{q'}} \left| \mathcal{R}_h(v_k) \right| \, dx   \right. \\
& \quad +   \left. \into  |u_k - u_0|^{\gamma} \left| \mathcal{R}_h(u_k) \right|^{\frac{p}{q'}} \left| \mathcal{R}_h(v_k) \right| \, dx \right).  \\
\end{align*}
By applying the extended H\"older inequality with conjugated exponents $\displaystyle p^*/\gamma$, $\dis q'p^*/p$ and $q^*$, we have
$$ \displaystyle \into  |u_k - u_0|^{\gamma} \left| \mathcal{R}_h(u_k) \right|^{\frac{p}{q'}} \left| \mathcal{R}_h(v_k) \right| \, dx \leq \lVert  u_k - u_0 \rVert_{p^*}^\gamma \lVert \mathcal{R}_h(u_k) \rVert_{p^*}^{\frac{p}{q'}} \lVert \mathcal{R}_h(v_k) \rVert_{q^*}.$$
Moreover, by Young's inequality with conjugated exponents $q'$ and $q$, we have 
$$ \displaystyle \lVert \mathcal{R}_h(u_k) \rVert_{p^*}^{\frac{p}{q'}} \lVert \mathcal{R}_h(v_k) \rVert_{q^*} \leq \frac{1}{q'} 
\lVert \mathcal{R}_h(u_k) \rVert_{p^*}^{p} + \frac{1}{q} \lVert \mathcal{R}_h(v_k) \rVert_{q^*}^q,$$
whence
\begin{align*}
\displaystyle 
& c_{p,q} \into  |u_k - u_0|^{\gamma} \left| \mathcal{R}_h(u_k) \right|^{\frac{p}{q'}} \left| \mathcal{R}_h(v_k) \right| \, dx  \\
\leq  & \, c_{p,q} \lVert  u_k - u_0 \rVert_{p^*}^\gamma \left( \lVert \mathcal{R}_h(u_k) \rVert_{p^*}^{p} +  \lVert \mathcal{R}_h(v_k) \rVert_{q^*}^q  \right).
\end{align*}
Summing up, we obtained
\begin{align*}
\displaystyle
& \into  |u_k|^{p^*\frac{q^{*}-1}{q^{*}}} \left| \mathcal{R}_h(v_k) \right| \, dx  \\ 
\leq & c_{p,q} \left(\into  |u_k|^{\gamma} \left| \mathcal{T}_h(u_k) \right|^{\frac{p}{q'}} \left| \mathcal{R}_h(v_k) \right| \, dx  +  \into  |u_0|^{\gamma} \left| \mathcal{R}_h(u_k) \right|^{\frac{p}{q'}} \left| \mathcal{R}_h(v_k) \right| \, dx   \right) \\
& \; +  c_{p,q}\lVert  u_k - u_0 \rVert_{p^*}^\gamma \left( \lVert \mathcal{R}_h(u_k) \rVert_{p^*}^{p} +  \lVert \mathcal{R}_h(v_k) \rVert_{q^*}^q  \right) .  \\
\end{align*}
Observe now that the conjugated exponents $\displaystyle p^*/\gamma$, $\dis q'p^*/p$ and $q^*$, imply that $ |u_k(x)|^{\gamma} \left| \mathcal{T}_h(u_k(x)) \right|^{\frac{p}{q'}} \left| \mathcal{R}_h(v_k(x)) \right|  \leq w(x)$ for any $k$ and for almost every $x \in \Omega$, with $w \in L^1(\Omega)$. Therefore, by dominated convergence theorem we have 
$$ \displaystyle \into  |u_k|^{\gamma} \left| \mathcal{T}_h(u_k) \right|^{\frac{p}{q'}} \left| \mathcal{R}_h(v_k) \right| \, dx  \to \into  |u|^{\gamma} \left| \mathcal{T}_h(u) \right|^{\frac{p}{q'}} \left| \mathcal{R}_h(v) \right| \, dx.$$
Now, observe that the conjugated exponents $\displaystyle \frac{(p^*-\gamma)q^*}{(p^*-\gamma)q^*-p^*}$ and $\displaystyle \frac{(p^*-\gamma)q^*}{p^*}$ give that the sequence $\biggl\{ \left| \mathcal{R}_h(u_k) \right|^{\frac{p}{q'}} \left| \mathcal{R}_h(v_k) \right|    \biggr\}_k$ is bounded in $\displaystyle L^{\left( \frac{p^*}{\gamma} \right)'}(\Omega).$ Indeed, considering that there exist $\bar{u} \in L^{p^*}(\Omega)$ and $\bar{v} \in L^{q^*}(\Omega)$ such that $|u_k(x)| \leq \bar{u}(x)$ and $|v_k(x)| \leq \bar{v}(x)$ for any $k$ and for almost every $x \in \Omega$, we have 
\begin{align*}
\displaystyle 
\into \left(  \left| R_h(u_k) \right|^{\frac{p}{q'}} \left| R_h(v_k) \right| \right)^{\left( \frac{p^*}{\gamma} \right)'} \, dx  
&      \leq      \into  \left| \bar{u} \right|^{\frac{p}{q'} \frac{p^*}{p^*-\gamma}} \left| \bar{v} \right|^{\frac{p^*}{p^*-\gamma}} \, dx \\
& \leq       \lVert \bar{u}   \rVert_{p^*}^{\frac{p^*\left[(p^*-\gamma)q^*-p^*\right]}{(p^*-\gamma)q^*}} \lVert \bar{v}   \rVert_{q^*}^{\frac{p^*}{p^*-\gamma}}.  \\
\end{align*} 
By Proposition \ref{weakconvLp}, we deduce that $\left| \mathcal{R}_h(u_k) \right|^{\frac{p}{q'}} \left| \mathcal{R}_h(v_k) \right| \rightharpoonup  \left| \mathcal{R}_h(u) \right|^{\frac{p}{q'}} \left| \mathcal{R}_h(v) \right|$ in $\displaystyle L^{\left( \frac{p^*}{\gamma} \right)'}(\Omega).$ Since $|u_0|^{\gamma} \in L^{\frac{p^*}{\gamma} }(\Omega)$, we have
$$  \into  |u_0|^{\gamma} \left| \mathcal{R}_h(u_k) \right|^{\frac{p}{q'}} \left| \mathcal{R}_h(v_k) \right| \, dx  \to  \into  |u_0|^{\gamma} \left| \mathcal{R}_h(u) \right|^{\frac{p}{q'}} \left| \mathcal{R}_h(v) \right| \, dx. $$
Combining these results, we conclude that
\begin{align*}
\displaystyle
& \into  |u_k|^{p^*\frac{q^{*}-1}{q^{*}}} \left| \mathcal{R}_h(v_k) \right| \, dx  \\
 \leq &  c_{p,q} \left(\into  |u|^{\gamma} \left| \mathcal{T}_h(u) \right|^{\frac{p}{q'}} \left| \mathcal{R}_h(v) \right| \, dx  +  \into  |u_0|^{\gamma} \left| \mathcal{R}_h(u) \right|^{\frac{p}{q'}} \left| \mathcal{R}_h(v) \right| \, dx   \right) \\
& \quad +  c_{p,q} \lVert  u_k - u_0 \rVert_{p^*}^\gamma \left( \lVert \mathcal{R}_h(u_k) \rVert_{p^*}^{p} +  \lVert \mathcal{R}_h(v_k) \rVert_{q^*}^q  \right)  + o_k(1).  \\
\end{align*}
Arguing in a similar way on the term
$$ \displaystyle \into  |v_k|^{q^*\frac{p^{*}-1}{p^{*}}} \left| \mathcal{R}_h(u_k) \right| \, dx$$
and denoting by
$$ \displaystyle \delta := q^*\frac{p^{*}-1}{p^{*}} - \frac{q}{p'}=\frac{q^* \, p^* - q^* - \displaystyle \frac{q}{p'}p^*}{p^*}, $$
we get
\begin{align*}
\displaystyle
& \into  |v_k|^{q^*\frac{p^{*}-1}{p^{*}}} \left| \mathcal{R}_h(u_k) \right| \, dx  \\
\leq &  c_{p,q} \left(\into  |v|^{\delta} \left| \mathcal{T}_h(v) \right|^{\frac{q}{p'}} \left| \mathcal{R}_h(v) \right| \, dx  +  \into  |v_0|^{\delta} \left| \mathcal{R}_h(v) \right|^{\frac{q}{p'}} \left| \mathcal{R}_h(u) \right| \, dx   \right) \\
& \quad +   c_{p,q} \lVert  v_k - v_0 \rVert_{q^*}^\delta \left( \lVert \mathcal{R}_h(v_k) \rVert_{q^*}^{q} +  \lVert \mathcal{R}_h(u_k) \rVert_{p^*}^p  \right)  + o_k(1).  \\
\end{align*}
Altogether, denoting by $K:=K(p,q,C)$ the greatest positive constant in previous inequalities, by \eqref{calcoloPalaisSmale3} we deduce 
\begin{align*}
\displaystyle
\limsup_{k \to + \infty}  \; & \Biggl\{ \quad   \nu_1 \frac{S(p,N)}{p} \lVert \mathcal{R}_h(u_k) \rVert_{p^*}^{p}  - K \lVert u_k - u_0 \rVert_{p^*}^{p^* - p} \lVert \mathcal{R}_h(u_k) \rVert_{p^*}^p    \\
& \quad + \nu_2 \frac{S(q,N)}{q} \lVert \mathcal{R}_h(v_k) \rVert_{q^*}^{q} - K\lVert v_k - v_0 \rVert_{q^*}^{q^* - q} \lVert \mathcal{R}_h(v_k) \rVert_{q^*}^q  \\
& \quad  -K \lVert  u_k - u_0 \rVert_{p^*}^\gamma \left( \lVert \mathcal{R}_h(u_k) \rVert_{p^*}^{p} +  \lVert \mathcal{R}_h(v_k) \rVert_{q^*}^q  \right)  \\
& \quad -K\lVert  v_k - v_0 \rVert_{q^*}^\delta \left( \lVert \mathcal{R}_h(v_k) \rVert_{q^*}^{q} +  \lVert \mathcal{R}_h(u_k) \rVert_{p^*}^p  \right) \; \Biggr\} \quad  \\
\leq & \quad K \into |u|^{p^*-1} \left| \mathcal{R}_h(u) \right| \, dx + K \into |v|^{q^{*}\frac{p^*-1}{p^{*}}} \left| \mathcal{R}_h(u) \right| \, dx \\
& + K \into |u|^{p^*\frac{q^{*}-1}{q^{*}}} \left| \mathcal{R}_h(v) \right| \, dx + K \into |v|^{q^*-1} \left| \mathcal{R}_h(v) \right| \, dx \\
& + K \into \left| \mathcal{R}_h(u) \right| \, dx  + K \into \left| \mathcal{R}_h(v) \right| \, dx + o_h(1) \\
& + K \left( \into |u|^{p^*-p} \left| \mathcal{T}_h(u)  \right|^{p-1} \left| \mathcal{R}_h(u) \right| dx  + \into |u_0|^{p^*-p} \left| \mathcal{R}_h(u)  \right|^{p}   dx \right) \\
& + K \left( \into |v|^{q^*-q} \left| \mathcal{T}_h(v)  \right|^{q-1} \left| \mathcal{R}_h(v) \right| dx  + \into |v_0|^{q^*-q} \left| \mathcal{R}_h(v)  \right|^{q}  dx \right) \\
& + K \left(\into  |u|^{\gamma} \left| \mathcal{T}_h(u) \right|^{\frac{p}{q'}} \left| \mathcal{R}_h(v) \right| \, dx  +  \into  |u_0|^{\gamma} \left| \mathcal{R}_h(u) \right|^{\frac{p}{q'}} \left| \mathcal{R}_h(v) \right| \, dx   \right)        \\
& + K \left(\into  |v|^{\delta} \left| \mathcal{T}_h(v) \right|^{\frac{q}{p'}} \left| \mathcal{R}_h(v) \right| \, dx  + \into  |v_0|^{\delta} \left| \mathcal{R}_h(v) \right|^{\frac{q}{p'}} \left| \mathcal{R}_h(u) \right| \, dx   \right)       .\\
\end{align*}
Taking into account the left side of previous inequality, we have
\begin{align*}
\displaystyle
& \nu_1 \, \frac{S(p,N)}{p} \lVert \mathcal{R}_h(u_k) \rVert_{p^*}^{p}  - K \lVert u_k - u_0 \rVert_{p^*}^{p^* - p} \lVert \mathcal{R}_h(u_k) \rVert_{p^*}^p    \\
+ & \nu_2 \, \frac{S(q,N)}{q} \lVert \mathcal{R}_h(v_k) \rVert_{q^*}^{q} - K\lVert v_k - v_0 \rVert_{q^*}^{q^* - q} \lVert \mathcal{R}_h(v_k) \rVert_{q^*}^q  \\
- & K \lVert  u_k - u_0 \rVert_{p^*}^{\gamma} \left( \lVert \mathcal{R}_h(u_k) \rVert_{p^*}^{p} +  \lVert \mathcal{R}_h(v_k) \rVert_{q^*}^q  \right)  \\
- & K\lVert  v_k - v_0 \rVert_{q^*}^{\delta} \left( \lVert \mathcal{R}_h(v_k) \rVert_{q^*}^{q} +  \lVert \mathcal{R}_h(u_k) \rVert_{p^*}^p  \right)\\
= & \left[ \nu_1 \, \frac{S(p,N)}{p} - K \left( \lVert u_k - u_0 \rVert_{p^*}^{p^* - p}  + \lVert  u_k - u_0 \rVert_{p^*}^{\gamma} + \lVert  v_k - v_0 \rVert_{q^*}^{\delta} \right)    \right] \lVert \mathcal{R}_h(u_k) \rVert_{p^*}^{p} \\
+ & \left[\nu_2 \,  \frac{S(q,N)}{q} - K \left( \lVert v_k - v_0 \rVert_{q^*}^{q^* - q}  + \lVert  v_k - v_0 \rVert_{q^*}^{\delta} + \lVert  u_k - u_0 \rVert_{p^*}^{\gamma} \right)   \right] \lVert \mathcal{R}_h(v_k) \rVert_{q^*}^{q}.
\end{align*}
Now, let us recall that
$$ \displaystyle \overline{B_R(z_0)}:=\left\lbrace z \in X \; : \; \lVert z - z_0 \rVert \leq R \right\rbrace$$
where $R>0$ is a general radius.
Given $0<M < \min \left\lbrace \nu_1 \, \frac{S(p,N)}{p} , \nu_2 \, \frac{S(q,N)}{q}    \right\rbrace$ and considering that $\{z_k\}_k=\{(u_k,v_k)\}_k \subset \overline{B_R(z_0)} $, we can choose $R$ small enough such that for any $k$ we have
$$ \displaystyle  \nu_1 \, \frac{S(p,N)}{p} - K \left( \lVert u_k - u_0 \rVert_{p^*}^{p^* - p}  + \lVert  u_k - u_0 \rVert_{p^*}^{\gamma} + \lVert  v_k - v_0 \rVert_{q^*}^{\delta} \right) \geq M $$
and 
$$ \displaystyle \nu_2 \, \frac{S(q,N)}{q} - K \left( \lVert v_k - v_0 \rVert_{q^*}^{q^* - q}  + \lVert  v_k - v_0 \rVert_{q^*}^{\delta} + \lVert  u_k - u_0 \rVert_{p^*}^{\gamma} \right)\geq M. $$
Therefore,
\begin{align*}
\displaystyle
 M & \limsup_{k \to + \infty} \left\lbrace \lVert \mathcal{R}_h(u_k) \rVert_{p^*}^p \right.  + \left. \lVert \mathcal{R}_h(v_k) \rVert_{q^*}^q \right\rbrace \\
& \leq  K \into |u|^{p^*-1} \left| \mathcal{R}_h(u) \right| \, dx + K \into |v|^{q^{*}\frac{p^*-1}{p^{*}}} \left| \mathcal{R}_h(u) \right| \, dx \\
& + K \into |u|^{p^*\frac{q^{*}-1}{q^{*}}} \left| \mathcal{R}_h(v) \right| \, dx + K \into |v|^{q^*-1} \left| \mathcal{R}_h(v) \right| \, dx \\
& + K \into \left| \mathcal{R}_h(u) \right| \, dx  + K \into \left| \mathcal{R}_h(v) \right| \, dx + o_h(1) \\
& + K \left( \into |u|^{p^*-p} \left| \mathcal{T}_h(u)  \right|^{p-1} \left| \mathcal{R}_h(u) \right| dx  + \into |u_0|^{p^*-p} \left| \mathcal{R}_h(u)  \right|^{p}   dx \right) \\
& + K \left( \into |v|^{q^*-q} \left| \mathcal{T}_h(v)  \right|^{q-1} \left| \mathcal{R}_h(v) \right| dx  + \into |v_0|^{q^*-q} \left| \mathcal{R}_h(v)  \right|^{q}  dx \right) \\
& + K \left(\into  |u|^{\gamma} \left| \mathcal{T}_h(u) \right|^{\frac{p}{q'}} \left| \mathcal{R}_h(v) \right| \, dx  +  \into  |u_0|^{\gamma} \left| \mathcal{R}_h(u) \right|^{\frac{p}{q'}} \left| \mathcal{R}_h(v) \right| \, dx   \right)        \\
& + K \left(\into  |v|^{\delta} \left| \mathcal{T}_h(v) \right|^{\frac{q}{p'}} \left| \mathcal{R}_h(v) \right| \, dx  + \into  |v_0|^{\delta} \left| \mathcal{R}_h(v) \right|^{\frac{q}{p'}} \left| \mathcal{R}_h(u) \right| \, dx   \right)       .
\end{align*}
By applying dominated convergence theorem (with respect to $h$) to the right side of previous inequality, we obtain that 
$$ \displaystyle \lim_{h \to + \infty} \, \limsup_{k \to + \infty}  \left(\lVert \mathcal{R}_h(u_k) \rVert_{p^*}^p + \lVert \mathcal{R}_h(v_k) \rVert_{q^*}^q \right) =0,$$
by which we deduce
$$ \displaystyle \lim_{h \to + \infty} \, \limsup_{k \to + \infty}  \; \lVert \mathcal{R}_h(u_k) \rVert_{p^*}  = 0$$ 
and
$$ \displaystyle  \lim_{h \to + \infty} \, \limsup_{k \to + \infty}   \; \lVert \mathcal{R}_h(v_k) \rVert_{q^*} = 0.  $$ 
Now, observing that 
$$ \displaystyle \lVert u_k - u \rVert_{p^*} \leq \lVert \mathcal{T}_h(u_k) - \mathcal{T}_h(u) \rVert_{p^*} + \lVert \mathcal{R}_h(u_k) \rVert_{p^*} + \lVert \mathcal{R}_h(u) \rVert_{p^*},$$
we have
$$ \displaystyle \limsup_{k \to + \infty} \, \lVert u_k - u \rVert_{p^*} \leq  \limsup_{k \to + \infty} \, \lVert \mathcal{R}_h(u_k) \rVert_{p^*} + \lVert \mathcal{R}_h(u) \rVert_{p^*}.$$
By dominated convergence theorem, we have 
$$ \displaystyle \lim_{h \to +  \infty}  \, \lVert \mathcal{R}_h(u) \rVert_{p^*}=0, $$
hence we deduce 
$$ \displaystyle \lVert u_k - u \rVert_{p^*} \to 0 ,$$
and similarly
$$ \displaystyle \lVert v_k - v \rVert_{q^*} \to 0. $$
Finally, we infer that
$$ \displaystyle \into \left[ H_s(\delta,x,u_k,v_k) (u_k - u) + H_t(\delta, x,u_k,v_k)(v_k - v) \right] \, dx \to 0,$$
and coming back to \eqref{ipotesiHapbq} we deduce  
\begin{equation*}
\displaystyle \limsup_{k \to + \infty} \, \langle \mathcal{H}_{\Psi_1,\Psi_2}(z_k),z_k-z    \rangle
\leq 0.
\end{equation*}
By applying Proposition \ref{HPsi1Psi2S+}, we conclude that $\lVert z_k - z \rVert \to 0$.
\qed

\bigskip
\section{PROOF OF THEOREM ~\ref{TheorFinGrupCrit}.~}\label{finitezzagruppicritici}

\medskip
\noindent
In this Section, we prove Theorem \ref{TheorFinGrupCrit}, that is each isolated critical point of the functional $I_{\delta,\Psi_1,\Psi_2}$ defined in \eqref{funzionale0} and associated to system \eqref{AnisotropicSystem}, has critical groups of finite type and that the  Poincaré-Hopf formula holds.
As mentioned in the Introduction, the proof is a straightforward consequence of the Theorem \ref{IprimodeltaPsi1Psi2S+} and \cite[Theorem 1.1]{CD1}. However, for the reader's convenience  we retrace the main ideas from  \cite[Theorem 1.1]{CD1} that yield the result.

\medskip
\noindent
{\mbox {\it Proof of Theorem~\ref{TheorFinGrupCrit}.~}} 
Let $\bar{z}$ be an isolated critical point of the functional $I_{\delta,\Psi_1,\Psi_2}$ defined in \eqref{funzionale0}. Moreover, taking into account that $I_{\delta,\Psi_1,\Psi_2}: X \to \mathbb{R}$ is of class $C^1$ and that, by Theorem \ref{IprimodeltaPsi1Psi2S+}, its derivative is locally of class $(S)_+$, there exist $L,R>0$ such that $I_{\delta,\Psi_1,\Psi_2}$ is Lipschitz continuous with constant $L$ on $\overline{B_{4R}(\bar{z})}$, $I_{\delta,\Psi_1,\Psi_2}'$ is of class $(S)_+$ on $\overline{B_{4R}(\bar{z})}$, and $I_{\delta,\Psi_1,\Psi_2}'(z) \neq 0$ for every $z \in \overline{B_{4R}(\bar{z})} \setminus \{ \bar{z} \}$. In particular,  by applying \cite[Lemma 4.1, Lemma 4.2 and Lemma 4.3]{CD1}, there exist $\sigma \in [0,L]$ and a finite dimensional subspace $Y$ of $X$, with $\bar{z} \in Y$, such that, setting 
\begin{align*}
\displaystyle
\upbeta(z) & := I_{\delta,\Psi_1,\Psi_2}(\bar{z}) + \sigma \left( R  - ( \lVert z - \bar{z} \rVert - R)^+    \right),\\
\Sigma   & :=  \left\lbrace z \in X \, : \quad I_{\delta,\Psi_1,\Psi_2}(\bar{z}) - 2 \sigma R \leq I_{\delta,\Psi_1,\Psi_2}(z) \leq \upbeta(z) \right\rbrace ,\\
A        & := \left\lbrace z \in \Sigma \, : \quad I_{\delta,\Psi_1,\Psi_2}(z) \leq I_{\delta,\Psi_1,\Psi_2}(\bar{z}) - \sigma R \right\rbrace,
\end{align*}
the following facts hold:
\begin{itemize}
\item[•] $H_{*} ( \Sigma, A)$ is of finite type;
\item[•] for any  $m \geq 0$, we have
\begin{equation*}
\displaystyle H_{m} ( \Sigma, A) \approx H_{m} ( \Sigma \cap Y, A \cap Y).
\end{equation*}
\end{itemize}
By applying  \cite[Proposition 3.2]{CD1} and \cite[Proposition 4.1]{CD1}, we also deduce 
\begin{equation*}
\mathrm{deg}\left((I_{\delta,\Psi_1,\Psi_2}\bigl|_{U \cap Y} )',B_R(\bar{z}) \cap Y,0 \right) = \sum_{m\geq 0}
(-1)^m\,\rnk{H_{m} ( \Sigma \cap Y, A \cap Y)} 
\end{equation*}
and
\begin{equation*}
\displaystyle C_m(I_{\delta,\Psi_1,\Psi_2},\bar{z}) \approx H_m(\Sigma, A) \qquad \text{ for any } m \geq 0.
\end{equation*}
Moreover, the definition and properties of the degree for maps of class $(S)_+$ give that
\begin{equation*}
\mathrm{deg}\left(I_{\delta,\Psi_1,\Psi_2}',B_R(\bar{z}),0 \right)=\mathrm{deg}\left((I_{\delta,\Psi_1,\Psi_2}\bigl|_{U \cap Y} )',B_R(\bar{z}) \cap Y,0 \right).
\end{equation*}
Altogether, we infer that
\begin{equation*}
\mathrm{deg}\left(I_{\delta,\Psi_1,\Psi_2}',B_R(\bar{z}),0 \right) = \sum_{m\geq 0}
(-1)^m\,\rnk{ C_m(I_{\delta,\Psi_1,\Psi_2},\bar{z})  }.
\end{equation*}
\qed

\bigskip
\noindent
{\bf Acknowledgments.}\\
The author is supported by INdAM-GNAMPA and INdAM-GNAMPA project {\sl “Metodi variazionali e topologici tra Fisica, Geometria e Scienze Applicate''} (CUP E53C25002010001).

\end{document}